\documentclass[12pt,reqno]{amsart}
\usepackage{fullpage}   
\usepackage{amsmath,amsthm,amssymb,array,enumitem,url,verbatim,xcolor,mathtools,multicol,enumitem,cleveref}
\usepackage{tikz}
\pgfdeclarelayer{background}
\pgfdeclarelayer{foreground}
\pgfsetlayers{background,main,foreground}
\theoremstyle{plain}
\newtheorem{theorem}{Theorem}
\newtheorem{lemma}[theorem]{Lemma}
\newtheorem{definition}[theorem]{Definition}
\newtheorem{proposition}[theorem]{Proposition}
\newtheorem{corollary}[theorem]{Corollary}

\theoremstyle{remark}
\newtheorem{example}[theorem]{Example}

\newtheorem*{acknowledgment}{Acknowledgment}

\numberwithin{equation}{section}
\numberwithin{theorem}{section}

\newcommand{\ldiv}{\backslash}
\newcommand{\rdiv}{/}

\newcommand{\Aut}{\mathrm{Aut}}

\newcommand{\Mlt}{\mathrm{Mlt}}
\newcommand{\Inn}{\mathrm{Inn}}
\newcommand{\chr}{\mathrm{char}}

\title{Complete Graph Decompositions and P-Groupoids}

\author[J. Carr]{John Carr${}^*$}
\address{Department of Mathematics and Statistics\\
Auburn University\\
Auburn, AL 36849 USA}
\email{\url{jac0017@auburn.edu}}

\author[M. Greer]{Mark Greer}
\address{Department of Mathematics\\
One Harrison Plaza \\
University of North Alabama\\
Florence AL 35632 USA }
\email{\url{mgreer@una.edu}}

\subjclass[2010]{20N05, 05C25}
\keywords{Complete Graphs, Hamiltonian Decompositions, P-Groupoids, P-Quasigroups, Quandles}
\thanks{${}^*$Research partially supported by University of North Alabama QEP Undergraduate Research Grant}

\begin{document}
\allowdisplaybreaks
\begin{abstract}
We study P-groupoids that arise from certain decompositions of complete graphs.  We show that left distributive P-groupoids are distributive, quasigroups.  We characterize P-groupoids when the corresponding decomposition is a Hamiltonian decomposition for complete graphs of odd, prime order.  We also study a specific example of a P-quasigroup constructed from cyclic groups of odd order.  We show such P-quasigroups have characteristic left and right multiplication groups, as well as the right multiplication group is isomorphic to the dihedral group.
\end{abstract}

\maketitle

\section{Introduction}
\label{intro}

The concept of graph amalgamation was introduced in 1984 by Anthony Hilton \cite{hilton84}. Recently, the subject has gained more attention and is becoming more widely studied. We aim to provide insight into graph amalgamation by considering the results of amalgamation in Latin squares. First, we cover some preliminaries.

Recall that a graph is an ordered pair $G=(V,E)$ comprising a set $V$ of vertices with a set $E$ of edges.  A \emph{complete graph}, denoted by $K_n$ where $n$ is the number of vertices in the graph, is a graph where every pair of vertices is connected by an edge.  An \emph{edge coloring} of a graph $G$ is a function $\gamma: C\to E(G)$, where $C$ is a set of colors.  A \emph{Hamiltonian decomposition} of $K_{2n+1}$ is an edge-coloring of $K_{2n+1}$ with $n$ colors in which each color class is a $C_{2n+1}$ cycle, called \emph{Hamiltonian cycles}. 
\begin{example}
Consider $K_5$ and its Hamiltonian decomposition.
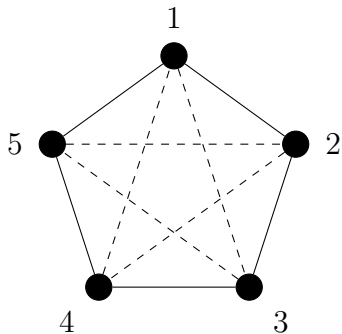
\begin{figure}[h!]
\null
\begin{tikzpicture}
\node[draw,circle,fill,inner sep=3.5pt] (1) at (0,3.077) {};
\node[draw,circle,fill,inner sep=3.5pt] (2) at (1.618,1.902) {};
\node[draw,circle,fill,inner sep=3.5pt] (3) at (1,0) {};
\node[draw,circle,fill,inner sep=3.5pt] (4) at (-1,0) {};
\node[draw,circle,fill,inner sep=3.5pt] (5) at (-1.618,1.902) {};
\node [above=.225cm] at (1) {1};
\node [right=.25cm] at (2) {2};
\node [below right=.175cm] at (3) {3};
\node [below left=.175cm] at (4) {4};
\node [left=.25cm] at (5) {5};
\begin{pgfonlayer}{background}
\draw[] (1) -- (2) -- (3) -- (4) -- (5) -- (1);
\draw[dashed] (1) -- (3) -- (5) -- (2) -- (4) -- (1);
\end{pgfonlayer}
\end{tikzpicture}
\end{figure}
\end{example}

We define \emph{graph amalgamation} in the following way.

\begin{definition}
Let $G$ and $H$ be two graphs with the same number of edges where $G$ has more vertices than $H$. We say that $H$ is an amalgamation of $G$ if there exists a bijection $\phi:E(G)\rightarrow E(H)$ and a surjection $\psi:V(G)\rightarrow V(H)$ where the following hold
\begin{enumerate}
\item If $x,y$ are two vertices in $G$ where $\psi(x)\neq\psi(y)$, and both $x$ and $y$ are adjacent by edge $e$ in $G$, then $\psi(x)$ and $\psi(y)$ are adjacent by edge $\phi(e)$ in $H$.
\item If $e$ is a loop on a vertex $x\in V(G)$, then $\phi(e)$ is a loop on $\psi(x)\in H$.
\item If $e$ joins $x,y\in V(G)$ where $x\neq y$, but $\psi(x)=\psi(y)$, then $\psi(e)$ is a loop on $\psi(x)$.
\end{enumerate}
\end{definition}

\begin{example} The following is an example of a graph amalgamation of the complete graph on 5 vertices with the amalgamation $\psi(1)=1$, $\psi(2)=2$, $\psi(3)=2$, $\psi(4)=3$, $\psi(5)=3$.
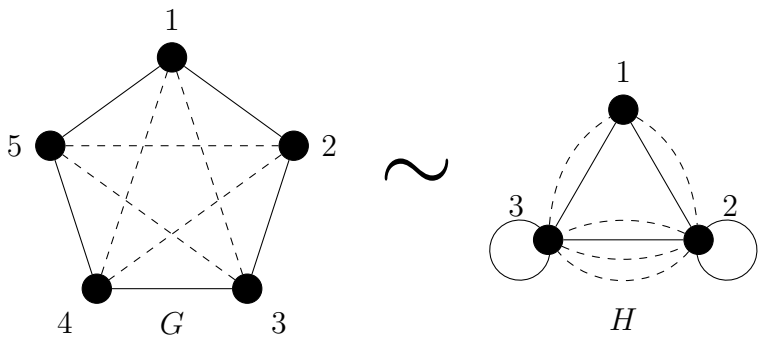
\begin{figure}[!h]
\centering
\begin{tikzpicture}
\node[draw,circle,fill] (1) at (0,3.077) {};
\node[draw,circle,fill] (2) at (1.618,1.902) {};
\node[draw,circle,fill] (3) at (1,0) {};
\node[draw,circle,fill] (4) at (-1,0) {};
\node[draw,circle,fill] (5) at (-1.618,1.902) {};
\node[draw,circle,fill] (6) at (6,2.382) {};
\node[draw,circle,fill] (7) at (7,.65) {};
\node[draw,circle,fill] (8) at (5,.65) {};
\node [above=.225cm] at (1) {1};
\node [below=3.25cm] at (1) {$G$};
\node [right=.225cm] at (2) {2};
\node [below right=.175cm] at (3) {3};
\node [below left=.175cm] at (4) {4};
\node [left=.225cm] at (5) {5};
\node [above=.225cm] at (6) {1};
\node [below=2.5cm] at (6) {$H$};
\node [above right=.175cm] at (7) {2};
\node [above left=.175cm] at (8) {3};
\node [draw=none,scale=3] (A) at (3.25,1.5) {$\sim$};

\begin{pgfonlayer}{background}
\draw[] (1) -- (2) -- (3) -- (4) -- (5) -- (1);
\draw[dashed] (1) -- (3) -- (5) -- (2) -- (4) -- (1);

\draw[] (6) -- (7) -- (8) -- (6);
\draw[] (8) arc [radius=.4cm, start angle=20, end angle= 380];
\draw[] (7) arc [radius=.4cm, start angle=160, end angle= 520];
\draw[dashed] 
(8) edge [out=85,in=210,distance=.6cm] (6) edge [out=335,in=205,distance=.55cm] (7) (6) edge [out=330,in=95,distance=.6cm] (7) (7) edge [out=155,in=25,distance=.55cm] (8) edge [out=230,in=310,distance=.675cm] (8);
\end{pgfonlayer}
\end{tikzpicture}
\end{figure}
\end{example}

Note that since the edges between two amalgamated graphs are in bijection with each other, edge colorings are invariant to amalgamation; that is, edge colors are unchanged by amalgamation. However, more interesting is the fact that if $G$ is a complete graph of the form $K_{2n+1}$ and the edges are colored in such a way as to specify a Hamiltonian decomposition, then the edges also form a Hamiltonian decomposition in $H$.

The concept of amalgamating a larger graph down into a smaller graph is a well understood concept in graph theory. Likewise, one can \emph{disentangle} vertices of a graph to create a larger graph. To disentangle a vertex is to split the vertex into multiple vertices. Using example 3, we could disentangle vertex 2 of graph $H$ into vertices 2 and 3, while disentangling vertex 3 into vertices 4 and 5 to create graph $G$. Some graph theorists are currently studying how to take a graph with a Hamiltonian decomposition such as graph $G$, and to disentangle $G$ to create a new graph, say $G'$, where $G'$ also has a Hamiltonian decomposition. Since the concept of amalgamation also exists in the Latin square setting, we approach the problem from an algebraic perspective.

Let $K_n$ be a complete graph.  It is well known that the edges in $K_n$ can be decomposed into distinct cycles if and only if $n$ is odd \cite{kotzig70}.  In this setting, Kotzig gave a complete characterization of a groupoid (termed P-groupoid) that would describe the decomposition.  Indeed, let $Q$ be a set with $n$ elements (corresponding to the vertices in $K_n$) and define $xy=z$ if and only if edges $(x,y)$ and $(y,z)$ are in the same cycle where $x\neq y$.  If $x=y$, then set $x^2=x$.

\begin{example}
Consider the previous example of $K_5$, along with its associated P-groupoid.
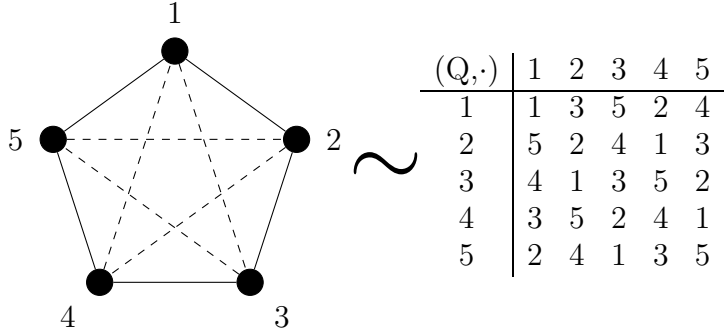
\begin{figure}[h!]
\parbox[s]{5cm}{
\null
\begin{tikzpicture}
\node[draw,circle,fill,inner sep=3.5pt] (1) at (0,3.077) {};
\node[draw,circle,fill,inner sep=3.5pt] (2) at (1.618,1.902) {};
\node[draw,circle,fill,inner sep=3.5pt] (3) at (1,0) {};
\node[draw,circle,fill,inner sep=3.5pt] (4) at (-1,0) {};
\node[draw,circle,fill,inner sep=3.5pt] (5) at (-1.618,1.902) {};
\node [above=.225cm] at (1) {1};
\node [right=.25cm] at (2) {2};
\node [below right=.175cm] at (3) {3};
\node [below left=.175cm] at (4) {4};
\node [left=.25cm] at (5) {5};
\node [draw=none,scale=3] (A) at (2.8,1.5) {$\sim$};
\begin{pgfonlayer}{background}
\draw[] (1) -- (2) -- (3) -- (4) -- (5) -- (1);
\draw[dashed] (1) -- (3) -- (5) -- (2) -- (4) -- (1);
\end{pgfonlayer}
\end{tikzpicture}
}
\parbox[s]{5cm}{
\null
\centering
\begin{tabular}{c|ccccc}

(Q,$\cdot$)& 1 & 2 & 3 & 4 & 5\\
\hline                      
1 & 1 & 3 & 5 & 2 & 4\\
2 & 5 & 2 & 4 & 1 & 3\\
3 & 4 & 1 & 3 & 5 & 2\\
4 & 3 & 5 & 2 & 4 & 1\\
5 & 2 & 4 & 1 & 3 & 5\\

\end{tabular}
}
\end{figure}

\end{example}

\newpage

Kotzig then showed that all decompositions of complete graphs are given by \emph{P-groupoids}, defining them as follows.
\begin{definition}[\cite{kotzig70}]
Let $(Q,\cdot)$ be a groupoid.  Then $(Q,\cdot)$ is a P-groupoid if for all $x,y,z\in Q$,
\begin{itemize}
\item[(\ref{pgroup}.1)] $x^2=x$ (Idempotent). \label{P1}
\item[(\ref{pgroup}.2)] $x\neq y \Rightarrow xy\neq x$  and $xy\neq y$. \label{P2}
\item[(\ref{pgroup}.3)] $xy=z \Leftrightarrow zy=x.$ \label{P3}
\end{itemize}
\label{pgroup}
\end{definition}
One can quickly show that the order of every P-groupoid is odd \cite{kotzig70} and that the equation $xa=b$ is always uniquely solvable for $x$.  Indeed, $xa=b \Leftrightarrow ba=x$.  Hence, P-groupoids are idempotent, right quasigroups.  We show that $\Mlt_\rho(Q)\ \chr\ \Mlt(Q)$ (Lemma \ref{char}) and that if the P-groupoid is left distributive, then it is right distributive and a quasigroup (Theorem \ref{leftdistthm}).  

D\'{e}nes and Keedwell gave the first specific example of a P-quasigroup relating to the decomposition \cite{DK}.  We also note that this P-quasigroup is a quandle and use results from \cite{macquarrie11} to describe the right multiplication group and automorphism group of D\'{e}nes and Keedwell's example (Theorems \ref{dihedral} \& \ref{autogroup}).  We then show that if $H\leq Q$ is a subquasigroup, then $|H|$ must divide $|Q|$ (Theorem \ref{sub}).  If the graph has prime order, then D\'{e}nes and Keedwell's example is exactly the P-quasigroup relating the Hamiltonian decomposition (Corollary \ref{prime}).  This immediately gives us that if $p$ is an odd prime, then the Hamiltonian decomposition is uniquely given by their example. 
\section{P-groudpoids and Quasigroups}
\label{back}  
A \emph{groupoid} $(Q,\cdot)$ is a set $Q$ with a binary operation $\cdot:Q\times Q \to Q$.  A \emph{quasigroup} $(Q,\cdot)$ is a groupoid such that for all $a,b\in Q$, the equations $ax = b$ and $ya = b$ have unique solutions $x,y\in Q$.  We denote these unique solutions by $x = a\ldiv b$ and $y = b\rdiv a$, respectively.  Standard references in quasigroup theory are \cite{bruck71, pflugfelder90}. All groupoids (quasigroups) considered here are finite.

To avoid excessive parentheses, we use the following convention:
\begin{itemize} 
\item multiplication $\cdot$ will be less binding than divisions $/,\backslash$.
\item divisions are less binding than juxtaposition.
\end{itemize}
For example $xy/z\cdot y\backslash xy$ reads as $((xy)/z)(y\backslash(xy))$.

For $x\in Q$, where $Q$ is a quasigroup, we define the \emph{right} and \emph{left translations} by $x$ by, respectively, $yR_x = yx$ and $yL_x = xy$ for all $y\in Q$. The fact that these mappings are permutations of $Q$ follows easily from the definition of a quasigroup.  It is easy to see that $yL_{x}^{-1}=x\backslash y$ and $yR_{x}^{-1}=y/x$.  We define the \emph{left multiplication group of Q}, $\Mlt_{\lambda}(Q)=\langle L_{x} \mid \forall x\in Q\rangle $, the \emph{right multiplication group of Q}, $\Mlt_{\rho}(Q)=\langle R_{x} \mid \forall x\in Q\rangle $ and the \emph{multiplication group of Q}, $\Mlt(Q)=\langle Mlt_\lambda(Q), \Mlt_\rho(Q)\rangle$.
\begin{lemma}
Let $Q$ be a P-groupoid.  Then $|R_{x}|=2$ for all $x\in Q$.  
\label{R2}
\end{lemma}
\begin{proof}
Let $|Q|=2n+1$ for some $n\in \mathbb{Z}$ and suppose $q_1x=q_2$ for some $x,q_1,q_2\in Q$.  Then $q_1R^2_{x}=q_2R_{x}=q_1$.  Moreover, $xR_{x}=x$.  Hence, $R_{x}=(x)(q_1q_2)(q_3q_4)\ldots(q_{2n})(q_{2n+1}).$  The desired result follows.
\end{proof}
Let $H$ be a subgroup of $G$.  If $H$ is generated by elements of the same order, \emph{i.e.} $H=\langle g_1,g_2,\ldots g_k | |g_1|=|g_2|=\ldots |g_k| = n\rangle$, then $H$ is a characteristic subgroup of $G$ ($H\ \chr \ G$).  Therefore, we have the following:
\begin{lemma}
Let $Q$ be a P-groupoid.  Then $\Mlt_\rho(Q)\ \chr \ \Mlt(Q)$.
\label{char}
\end{lemma}
\begin{proof}
Since every right translation of $Q$ has order 2, the results follows \cite{Hilton}
\end{proof}
A groupoid $Q$ is \emph{left distributive} if it satisfies $x(yz) = (xy)(xz)$ for all $x,y,z\in Q$. Similarly, it is \emph{right distributive} if it satisfies $(yz)x = (yx)(zx)$. A \emph{distributive} grouppoid is a groupoid that is both left and right distributive.
\label{leftdistthm}
\begin{theorem}
Let $Q$ P-groupoid.  If $Q$ is left distributive, then $Q$ is a distributive quasigroup. 
\end{theorem}
\begin{proof}
Let $Q$ be a left distributive, P-groupoid.  Note that by left distributivity, we have $x\cdot yx = xy\cdot x$.  Suppose that $xa=xb$ for some $x,a,b\in Q$.  Then we compute
\begin{align*}
(ax)(ab\cdot x)&= [(ax)(ab)](ax\cdot x) &\textnormal{ by left distributivity,}\\
&= [(ax)(ab)]a &\textnormal{ by Lemma \ref{R2},}\\
&= [(a\cdot xb)] a &\textnormal{ by left distributivity,}\\
&= a (xb\cdot a),\\
&= a (xa \cdot a) &\textnormal{ by assumption,}\\
&= ax &\textnormal{ by Lemma \ref{R2}.6}
\end{align*}
Hence, we have $ab\cdot x =ax$ by (\ref{P2}.2).  Thus, $ab=a$ and hence, $b=a$ by (\ref{P2}.2) again.  Thus, $Q$ is a quasigroup.

For right distributive, we first compute note that by left distributivity $x(xy\cdot z) =(xy\cdot y)(xy\cdot z)=(xy)(yz)$.  Using (\ref{P3}.3), we have
\begin{equation}
[x(xy\cdot z)](yz) = xy.
\label{e39}
\end{equation}
Similarly, $(xy\cdot z)x = (xy\cdot z)(xy\cdot y) = (xy)(zy)$ and $x(xy\cdot z) = (xy\cdot y)(xy\cdot z) = (xy)(yz)$ both by left distributivity again, thus
\begin{equation}
(xy\cdot z)x = (xy)(zy),
\label{e31}
\end{equation}
\begin{equation}
x(xy\cdot z) = (xy)(yz).
\label{e30}
\end{equation}

Hence we have 
\begin{align*}
(x\cdot yz)(xz\cdot u)&=[(xy)(xz)](xz\cdot u) &\textnormal{ by left distributivity,}\\
&=(xy)[(xy)(xz)\cdot u] &\textnormal{ by \eqref{e30} with } x\to xy, y\to yz, z\to u,\\
&=(xy)[(x\cdot yz)\cdot u] &\textnormal{ by left distributivity,}
\end{align*}
thus
\begin{equation}
(x\cdot yz)(xz\cdot u) = (xy)[(x\cdot yz) \cdot u].
\label{e45}
\end{equation}
Substituting $y\to yz$ in \eqref{e39} give $x(yz) = [x(x(yz)\cdot z)][yz\cdot z] = x(x(yz)\cdot z)\cdot y$.  So
\begin{equation}
x(yz) = x(x(yz)\cdot z)\cdot y.
\label{e120}
\end{equation}
Hence we compute
\begin{align*}
x &= [x\cdot x(yz)][\underbracket[.75pt]{x(yz)}],\\
&=[x\cdot x(yz)][x(x(yz)\cdot z)\cdot y] &\textnormal{ by \eqref{e120}},\\
&=[x(x(yz)\cdot z)][xz\cdot y] &\textnormal{ by \eqref{e45} with } y\to x(yz), u\to y.
\end{align*}
Thus
\begin{equation}
x = [x(x(yz)\cdot z)][xz\cdot y].
\label{e571}
\end{equation}
Replacing $x\to xy$ and $ y\to x$ in \eqref{e571} gives
\begin{align*}
xy&= [(xy)\cdot (xy\cdot xz)z][\underbracket[.75pt]{(xy\cdot z)x]} &\textnormal{ by \eqref{e571} with } x\to xy,\\
&=[(xy)\cdot (xy\cdot xz)z](xy\cdot zy) &\textnormal{ by \eqref{e31}},\\
&=[(xy)\cdot (x\cdot yz)z](xy\cdot zy) &\textnormal{ by left distributivity,}
\end{align*}
and therefore
\begin{equation}
xy = [(xy)\cdot (x\cdot yz)z](xy\cdot zy).
\label{e642a}
\end{equation}
Recalling \eqref{e30} and substituting $y\to yz$, we have $(x\cdot yz)y = (x\cdot yz)(yz\cdot z) = x\cdot (x\cdot yz)z$, so
\begin{equation}
(x\cdot yz)y = x\cdot (x\cdot yz)z.
\label{e54}
\end{equation}  
We compute
\begin{align*}
x(yz\cdot x) &= (x\cdot yz)x,\\
&=(xy \cdot xz)x &\textnormal{ by left distributivity},\\
&=(xy)\cdot (xy\cdot xz)z &\textnormal{ by \eqref{e54} } x\to xy, y\to x,
\end{align*}
and hence
\begin{equation}
x(yz\cdot x) = (xy)[(x\cdot yz)z].
\label{e173}
\end{equation}
Hence, the right hand side of \eqref{e642a} can be rewritten as
\begin{equation}
xy = [x(yz\cdot x)](xy\cdot zy).
\label{e642b}
\end{equation}
Using left distributivity, we have
\begin{align*}
(xy)\cdot (xz\cdot y)(zy) &= [(xy)(xz\cdot y)] (xy\cdot zy) &\textnormal{ by left distributivity},\\
&=[(xy\cdot xz)(xy\cdot y)] (xy\cdot zy) &\textnormal{ by left distributivity},\\
&=[(x\cdot yz)(xy\cdot y)](xy\cdot zy) &\textnormal{ by left distributivity},\\
&=[(x\cdot yz)x] (xy\cdot zy),\\
&=[x(yz\cdot x)](xy\cdot zy),
\end{align*}
and thus
\begin{equation}
[x(yz\cdot x)][(xy\cdot zy)] = (xy)[(xz\cdot y)(zy)].
\label{e236}
\end{equation}
Therefore, the right hand side of \eqref{e642b} can be rewritten as $xy=(xy)[(xz\cdot y)(zy)]$

Finally, since $(xy)[(xz\cdot y)(zy)] = xy$, we have $(xz\cdot y)(zy)=xy$ by (\ref{P2}.2) and thus $(xy)(zy)=xz\cdot y$ by (\ref{P3}.3).

\end{proof}

We now focus on the first specific constructions of a P-quasigroup dealing with Hamiltonian decompositions was given by De\'{n}es and Keedwell \cite{DK}.  
\begin{theorem}[\cite{DK}]
Consider $\mathbb{Z}_n=\{0,1,\ldots,n-1\}$ where $n=2k+1$ for some $k\in \mathbb{Z}$.  Define $r\circ s = 2s-r \mod n$.  Then $(\mathbb{Z}_n,\circ)$ is a P-quasigroup of order $n$.  
\label{DKex}
\end{theorem}
For the remainder of this section, $(\mathbb{Z}_n,\circ)$ will always refer to the quasigroup in Theorem \ref{DKex}.
A quasigroup $Q$ is medial (or entropic) if $(xy)(zw)=(xz)(yw)$ for all $x,y,z,w\in Q$.  Idempotent medial quasigroups are distributive \cite{stein56}.  There is a well-known correspondence between abelian groups and medial quasigroups, the Toyoda-Bruck theorem.

\begin{theorem}[\cite{stanovsky15}]
$(Q,\cdot)$ is a medial quasigroup if and only if there is an abelian group $(Q,+)$ such that $x\cdot y = f(x)+g(y)+c$ for all $x,y\in Q$ for some  commuting $f,g \in \Aut(Q)$ and $c\in Q$.
\end{theorem}
Note that if $(G,+)$ is an abelian group of odd order, then both $f(x)=-x$ and $g(y)=2y$ are automorphisms of $G$.  Hence, De\'{n}es and Keedwell's P-quasigroup is precisely the medial quasigroup of the form $x\circ y = f(x)+g(y)+0$.  
\begin{lemma}
$(\mathbb{Z}_n,\circ)$ is medial.
\end{lemma}

\begin{proposition}
For $(\mathbb{Z}_n,\circ)$, the following hold:
\begin{itemize}
\item[(i)] $yL_x^n=2^n(y-x)+x$ for all $x,y\in Q$.
\item[(ii)] $|L_x| = k$ where $k$ is the smallest integer such that $2^k\equiv 1 \mod n$.
\item[(iii)] $L_x^nR_x = R_xL_x^n$.
\end{itemize}
\label{p1}
\end{proposition}
\begin{proof}
Let $x,y\in Q$.  For $(i)$, $yL_x = 2y-x = 2(y-x)+x$. By induction, 
\[yL_x^{n+1} =(2y-x)L_x^n =  2^n((2y-x)-x)+x = 2^{n+1}(y-x)+x.\]  
For $(ii)$, let $k>0$ be the smallest integer such that $yL_x^k = y$.  Then, by $(1)$, 
 \[2^k(y-x)+x\equiv y  \Leftrightarrow 2^ky-y-2^kx+x \equiv 0\Leftrightarrow (y-x)(2^k-1)\equiv 0.\]
Hence, $2^k\equiv 1\mod n$.  Finally, 
\[ yL_xR_x=(2y-x)R_x = 3x-2y = (2x-y)L_x=yR_xL_x. \]
Since $\Mlt(Q)$ is a group, $(iii)$ follows.
\end{proof}
\begin{lemma}
Let $Q=(\mathbb{Z}_n,\circ)$.  Then $\Mlt_\lambda(Q)\ \chr \ \Mlt(Q)$.
\end{lemma}
\begin{proof}
This follows quickly from Proposition \ref{p1}.
\end{proof}
\begin{theorem}
Let $Q=(\mathbb{Z}_n,\circ)$.  Then $\Mlt_\rho(Q) \cong D_{2n}$, the dihedral group of order $2n$.
\label{dihedral}
\end{theorem}
\begin{proof}
Recall that $D_{2n}=\langle x,y\bigm| x^n=y^2=(xy)^2=1\ \forall x,y\in D_{2n}\rangle$. Let $x,y,z,l\in Q$. By proposition 3 we have that $|R_x|=2\ \forall x\in Q$ and thus $R_x^2=1$. Now, note that $zR_xR_{x+l}=(2x-z)R_{x+l}=2x+2l-(2x-z)=z+2l$. Likewise, $zR_{x+l}R_x=(2x+2l-z)R_x=2x-(2x+2l-z)=z-2l$.
\end{proof}
\begin{definition}
A groupoid $(Q,\cdot)$ is a quandle if
\begin{enumerate}
\item $a^2=a$ for all $a\in Q$,
\item For all $a,b\in Q$, the equations $xa = b$ have a unique solution,
\item $(ab)c=(ac)(bc)$ for all $a,b,c\in Q$.
\end{enumerate}
\end{definition}
Note that quandles are idempotent, right distributive, and right quasigroups.
\begin{lemma}
$(\mathbb{Z}_n,\circ)$ is a quandle.
\end{lemma}
Theorem \ref{dihedral} is well-known and $(\mathbb{Z}_n,\circ)$ is also referred to as the \emph{dihedral quandle} of order $n$.  Moreover, both $L_x$ and $R_y$ are affine maps for all $x,y$.  Indeed
\[
[(1-t)a+tb]L_x = 2[(1-t)a+tb]-x = (1-t)(2a-x)+j(2b-x) = (1-t)(aL_x)+t(bL_x),
\]
\[
[(1-t)a+tb]R_y = 2y - [(1-t)a+tb] =  (1-t)(2y-a) + t(2y-b)  = (1-t)(aR_y) + t(bR_y),
\]
for all $a,b,t\in \mathbb{Z}_n$.  That is

\begin{theorem}[\cite{macquarrie11}]
$\Aut(\mathbb{Z}_n,\circ)$ is isomorphic to the affine group Aff($\mathbb{Z}_n$).
\end{theorem} 
For a quandle $Q$, the \emph{inner automorphism group of Q}, $\Inn(Q)$ is the subgroup generated by $L_x$ for all $x\in Q$.  Thus, the authors of \cite{macquarrie11} restated Theorem \ref{dihedral} as $\Inn(\mathbb{Z}_n,\circ)$ is isomorphic to the dihedral group of order $n$.  Moverover, they showed that the automorphism group of $(\mathbb{Z}_n,\circ)$ is isomorphic to $\Mlt((\mathbb{Z}_n,\circ))$.
\begin{theorem}[\cite{macquarrie11}]
$\Aut(\mathbb{Z}_n,\circ)=\Mlt(\mathbb{Z}_n,\circ)$.
\label{autogroup}
\end{theorem}
Note that P-quasigroups always have subgroups $\langle x \rangle$ for all $x$.  It is well-known that in general, the order of a subquasigroup doesn't divide the order of the quasigroup.  However, for $(\mathbb{Z}_n,\circ)$, the order of the subquasigroup always divides the order of the quasigroup.
\begin{theorem}
Let $Q=(\mathbb{Z}_n,\circ)$.  If $H\leq Q$, then $|H|$ divides $|Q|$.  Hence, if $Q$ has prime order and $|H|\leq |Q|$, then $H=\langle x\rangle$ for some $x\in Q$ or $H=Q$.
\label{sub}
\end{theorem}
\begin{proof}
Let $H\leq Q$.  If $|H|=\langle x \rangle$, then $|H|=1$ and we are done.  Let $x,y\in Q$.  Then $y=x+k$, since both $x,y\in \mathbb{Z}_n$.  Then $x\circ y = x+2k\in H$.  Continuing, $x\circ (x+2k) = x+3k$, and thus, elements of $H$ are of the form $x+lk$.  Since $Q$ is finite, we must have $x+l_1k=x+l_2k$.  Thus, $k(l_1-l_2)\equiv 0\mod n$.  Thus, $k$ is a divisor of $n$.  Let $kl=n$.  Then $H=\{x,x+k,x+2k,\ldots x+(l-1)k\}$, and therefore $|H|=l$, a divisor of $n$.
\end{proof}
The following is a minimal example of a P-groupoid that is not a quandle, found by \textsc{Mace4} \cite{mccune09}.
\begin{example}
A P-groupoid of order 5 that is not a quandle.
\begin{center}
\begin{tabular}{c|ccccc}

(Q,$\cdot$)& 1 & 2 & 3 & 4 & 5\\
\hline                      
1 & 1 & 3 & 2 & 3 & 4\\
2 & 3 & 2 & 1 & 5 & 3\\
3 & 2 & 1 & 3 & 1 & 2\\
4 & 5 & 5 & 5 & 4 & 1\\
5 & 4 & 4 & 4 & 2 & 5\\

\end{tabular}
\end{center}
\end{example}
\section{Hamiltonian Decompostions and P-quasigroups}
\label{conc}
We now focus on the connection between P-quasigroups and Hamiltonian decompositions.  
\begin{theorem}
Let $Q_1$ and $Q_2$ be two P-groupoids.  Then, $Q_1\cong Q_2$ \emph{if and only if} the corresponding decompositions of the associated complete graph is isomorphic.
\label{iso}
\end{theorem}
\begin{proof}
Suppose $\phi$ is an isomorphism between $Q_1$ and $Q_2$ where $Q_1$ corresponds to a decomposition of $J_n$ and $Q_2$ corresponds to a decomposition of $K_n$. Recall for $a,b,c\in Q_1$ we say that edges $(a,b)(b,c)$ belong to the same cycle in $J_n$ if and only if $ab=c$. Then $(\phi(a),\phi(b))(\phi(b),\phi(c))$ belong to the same cycle in $K_n$ if and only if $\phi(a)\phi(b)=\phi(c)$. Since this is precisely how we establish a correspondence between P-groupoids and complete undirected graphs, we conclude that the decomposition of $J_n$ is isomorphic to the decomposition of $K_n$.

Alternatively, let $Q_1$ and $Q_2$ be P-groupoids where $Q_1$ corresponds to a decomposition of $J_n$ and $Q_2$ corresponds to a decomposition of $K_n$. Suppose $\phi$ is an isomorphism between the decomposition of $J_n$ and the decomposition of $K_n$. Recall for $a,b,c\in Q_1$ we say that edges $(a,b)(b,c)$ belong to the same cycle in $J_n$ if and only if $ab=c$. If $(\phi(a),\phi(b))(\phi(b),\phi(c))$ belong to the same cycle in $K_n$, then $\phi(a)\phi(b)=\phi(c)$. Again, since this is precisely how we establish a correspondence between P-groupoids and complete undirected graphs, we conclude that $Q_1\cong Q_2$.
\end{proof}
\begin{theorem}]\cite{DK}]
Let $p$ be an odd prime.  Then $(\mathbb{Z}_n,\circ)$ corresponds to a Hamiltonian decomposition in $K_p$.
\label{HD}
\end{theorem}
Combining Theorems \ref{iso} and \ref{HD}, we quickly derive the following.
\begin{corollary}
Let $p$ be an odd prime and $K_p$ the complete graph on $p$ vertices.  If $Q$ is a P-quasigroup corresponding to a Hamiltonian decomposition of $K_p$, then $Q\cong (\mathbb{Z}_n,\circ)$.  Moreover, $\Mlt_\rho(Q)\cong D_{2p}$.
\label{prime}
\end{corollary}
The following is motivated by Theorem \ref{sub}.
\begin{theorem}
Let $K_n$ have a Hamiltonian decomposition and let $Q$ be the corresponding P-groupoid.  Then $Q$ doesn't contain any nontrivial subgroupoids.
\end{theorem}
\begin{proof}
Let $|Q|=n$ correspond to a complete graph $K_n$ with a Hamiltonian decomposition. For the sake of contradiction, suppose $\exists H<Q$ where $|H|>1$. Since $H$ is a subgroupoid, $H$ is closed and multiplying the elements of $H$ will create a cycle with length less than $n$. However, this contradicts our assumption that $K_n$ has a Hamiltonian decomposition. Therefore, we conclude that $Q$ doesn't contain any subgroupoids with order greater than 1.
\end{proof}

The following example illustrates Theorem 21. In particular, if \(n\) is not prime, then the corresponding algebraic structure is no longer guaranteed to be a P-quasigroup, but is still a P-groupoid.

\begin{example}
A P-groupoid of order 9.
\begin{figure}[h!]
\parbox[s]{8cm}{
\null
\begin{tikzpicture}
\node[draw,circle,fill] (1) at (0.,5.6712818196177075) {};
\node[draw,circle,fill] (2) at (1.8793852415718169,4.987241532966371) {};
\node[draw,circle,fill] (3) at (2.879385241571816,3.255190725397493) {};
\node[draw,circle,fill] (4) at (2.5320888862379554,1.285575219373078) {};
\node[draw,circle,fill] (5) at (1,0) {};
\node[draw,circle,fill] (6) at (-1,0) {};
\node[draw,circle,fill] (7) at (-2.5320888862379562,1.2855752193730794) {};
\node[draw,circle,fill] (8) at (-2.879385241571816,3.2551907253974957) {};
\node[draw,circle,fill] (9) at (-1.8793852415718155,4.987241532966372) {};
\node [above=.25cm] at (1) {1};
\node [above right=.175cm] at (2) {2};
\node [right=.2cm] at (3) {3};
\node [below right=.175cm] at (4) {4};
\node [below=.2cm] at (5) {5};
\node [below=.2cm] at (6) {6};
\node [below left=.175cm] at (7) {7};
\node [left=.2cm] at (8) {8};
\node [above left=.175cm] at (9) {9};

\begin{pgfonlayer}{background}
\draw[] (1) -- (2) -- (3) -- (4) -- (5) -- (6) -- (7) -- (8) -- (9) -- (1);
\draw[densely dashed] (1) -- (3) -- (5) -- (2) -- (7) -- (4) -- (9) -- (6) -- (8) -- (1);
\draw[dotted] (1) -- (4) -- (2) -- (6) -- (3) -- (8) -- (5) -- (9) -- (7) -- (1);
\draw[loosely dashdotted] (1) -- (5) -- (7) -- (3) -- (9) -- (2) -- (8) -- (4) -- (6) -- (1);
\end{pgfonlayer}
\end{tikzpicture}
}
\parbox[s]{5cm}{
\null
\centering
\begin{tabular}{r|rrrrrrrrr}
$K_9$ & 1 & 2 & 3 & 4 & 5 & 6 & 7 & 8 & 9\\
\hline
1 & 1 & 3 & 5 & 2 & 7 & 4 & 9 & 6 & 8 \\
2 & 9 & 2 & 4 & 1 & 3 & 3 & 4 & 4 & 3 \\
3 & 8 & 1 & 3 & 5 & 2 & 2 & 5 & 5 & 2 \\
4 & 7 & 6 & 2 & 4 & 6 & 1 & 2 & 2 & 6 \\
5 & 6 & 7 & 1 & 3 & 5 & 7 & 3 & 3 & 7 \\
6 & 5 & 4 & 8 & 8 & 4 & 6 & 8 & 1 & 4 \\
7 & 4 & 5 & 9 & 9 & 1 & 5 & 7 & 9 & 5 \\
8 & 3 & 9 & 6 & 6 & 9 & 9 & 6 & 8 & 1 \\
9 & 2 & 8 & 7 & 7 & 8 & 8 & 1 & 7 & 9
\end{tabular}
}
\end{figure}

\end{example}

Further work would consist of finding all necessary and sufficient conditions such that a P-groupoid of odd nonprime order corresponds to a Hamiltonian decomposition of a complete graph.  Hilton gave necessary and sufficient conditions for a Hamiltonian decomposition of $K_{2n+1}$ corresponding to a Hamiltonian circuit \cite{hilton84}.  The proof relies heavily on Hall's work with completing partial Latin squares \cite{hall60}.  Thus, using P-groups to classify Hamiltonian decompositions is a natural choice.  Moreover, due to the connection to quandles in the prime order case, perhaps finding a relationship between P-groupoids and quandles could lead to new results in both fields.

\begin{acknowledgment}
Some investigations in this paper were assisted by the automated deduction tool, \textsc{Prover9}, and the finite model builder, \textsc{Mace4}, both developed by McCune \cite{mccune09}.  Similarly, all presented examples were found and verified using the GAP system \cite{GAP} together with the LOOPS package \cite{GAPNV}.  
\end{acknowledgment}

\bibliographystyle{amsplain}

\end{document}